\documentclass[12pt]{article}
\usepackage{graphicx,psfrag,subfigure,amsmath}
\input epsf
\usepackage{latexsym}
\usepackage{amsfonts, amssymb}
\usepackage{amsthm}
\usepackage{mathrsfs,color}
\usepackage{here}

\newtheorem{theorem}{Theorem}[section]
\newtheorem{conjecture}{Conjecture}[section]

\newtheorem{claim}{Claim}[section]

\newtheorem{lem}{Lemma}[section]
\newtheorem{corollary}[lem]{Corollary}

\newenvironment{wst}
{\setlength{\leftmargini}{1.5\parindent}
 \begin{itemize}
 \setlength{\itemsep}{-1.1mm}}
{\end{itemize}}
\begin{document}
\setlength{\baselineskip}{15pt}
\begin{center}{\bf \Large A {Fan-type} condition for cycles in $1$-tough and $k$-connected $(P_2\cup kP_1)$-free graphs \footnote{This work is financially supported by NSFC grants 11871239 and 11971196.}}
\vspace{4mm}

{\large Zhiquan Hu,\ \ Jie Wang,\ \ Changlong Shen\footnote{Corresponding author. \\
\hspace*{5mm}{\it Email addresses}: hu\_zhiq@aliyun.com (Z.Q. Hu),\ wangjie@mails.ccnu.edu.cn (J. Wang),\ clshen2019@foxmail.com (C.L. Shen).}}\vspace{2mm}

School of Mathematics and Statistics, and Key Laboratory of Nonlinear Analysis and Applications (Ministry of Education), Central China Normal University, Wuhan 430079, P. R. China
\end{center}

\date{}
\noindent {\bf Abstract}:
For a  graph $G$, let $\mu_k(G):=\min~\{\max_{x\in S}d_G(x):~S\in \mathcal{S}_k\}$, where $\mathcal{S}_k$ is the set consisting of all independent sets $\{u_1,\ldots,u_k\}$ of $G$ such that some vertex, say $u_i$ ($1\leq i\leq k$),  is at distance two from every other vertex in it.
A graph $G$ is $1$-tough if for each cut set $S\subseteq V(G)$, $G-S$ has at most $|S|$ components.  Recently, Shi and Shan \cite{Shi} conjectured that for each integer $k\geq 4$, being $2k$-connected is  sufficient for  $1$-tough $(P_2\cup kP_1)$-free graphs to be hamiltonian, which was confirmed by Xu et al. \cite{Xu} and Ota and Sanka \cite{Ota2}, respectively. In this article, we generalize the above results through the following Fan-type theorem: Let $k$ be an integer with $k\geq 2$ and let $G$ be a $1$-tough and $k$-connected $(P_2\cup kP_1)$-free graph with $\mu_{k+1}(G)\geq\frac{7k-6}{5}$, then $G$ is hamiltonian  or the Petersen graph.

\vskip 0.1cm


\section{Introduction}
In this paper, we consider only finite and simple graphs. The terminology not
defined here can be found in \cite{Diestel}.
For two integers $s$ and $t$ with $s\leq t$, define $[s,t] := \{i\in \mathbb{Z}:~ s\leq i\leq t\}$, $[s,t):=[s,t]\setminus\{t\}$ and $(s,t]:=[s,t]\setminus\{s\}$.

Let $G$ be a graph. We denote by $V(G)$, $E(G)$ and $\omega(G)$ the set of vertices, the set of edges, and the number of components of $G$, respectively. For $x\in V(G)$, we use $N_G(x)$ to denote the set of neighbors of $x$ in $G$ and define $d_G(x):=|N_G(x)|$.
The \textit{distance} between two vertices $x$ and $y$, denoted by $dist_G(x,y)$, is the length of a shortest path in $G$ between $x$ and $y$.
For $S\subseteq V(G)$, we define $G[S]$ as the subgraph of $G$ induced by $S$ and denote $G-S:=G[V(G)\setminus S]$. Define $N_G(S) :=\cup_{x\in S} N_G(x)$. We call $S$ a \emph{$k$-essential independent set of $G$ with center $x$} if $S$ is an independent set of $G$ with order $k$ such that $x\in S$ and $dist_G(x,v)=2$ for all $v\in S\setminus\{x\}$. The essential independence number of $G$, denoted by $\alpha_e(G)$, is the maximum integer $k$ such that $G$ contains a $k$-essential independent set. Note that if $\alpha_e(G)\geq k$ then there is a $k$-essential independent set in $G$.

For an integer $k\geq 2$, let $\mathcal{S}_k$ be the set  of all $k$-essential independent sets of $G$. Define
$$
\mu_k(G):=\left\{
            \begin{array}{ll}
              \min~\{\max_{x\in S}d_G(x): ~S\in \mathcal{S}_k\}, & \hbox{if $\alpha_e(G)\geq k$;} \\
              +\infty, & \hbox{otherwise.}
            \end{array}
          \right.
$$
Throughout this paper, we skip the subscript $G$ if no confusion may arise.

Let $R$ be a cycle or path of $G$. We always  assume that $R$ has a given direction.
For $x\in V(R)$, we use $x_R^{+}$ (resp. $x_R^{-}$) to denote its successor (resp. predecessor) on $R$. Define $x_R^{-2}=(x_R^-)_R^-$. If $X\subset V(R)$, then we set
$X_R^{-}:=\{x_R^{-}: x\in X\}$ and $X_R^{+}:=\{x_R^{+}: x\in X\}$.
For $x,y\in V(R)$, let $x\overrightarrow{R}y$ denote the path on $R$ from $x$ to $y$ in the chosen direction of $R$ if it exists. The same path, in the reverse order,
is denoted by $y\overleftarrow{R}x$. Let $R[x,y]=x\overrightarrow{R}y$ and $R[x,y)=x\overrightarrow{R}y_R^-$.
When necessary, we use $\overrightarrow{R}$ to emphasize the given direction of $R$.
If $H$ is a component of $G-V(R)$, then we let $N_R(H):=N_G(V(H))\cap V(R)$ and $N_R^-(H):=(N_R(H))_R^-$.
A path $P$ is an \emph{$(x,y)$-path} if it is a path from $x$ to $y$.

Let $t$  be a nonnegative real number. A graph $G$ is \emph{$t$-tough} if $|S|\geq t\cdot\omega(G-S)$ for each vertex cut $S$ of $G$.  A graph $G$ is \textit{hamiltonian} if it contains a \textit{Hamiltonian cycle}, i.e., a cycle containing every vertex of $G$.

In 1973, Chv\'{a}tal \cite{chvatal} proposed the following toughness conjecture.
\begin{conjecture}\label{cong-chv}
\emph{(}Chv\'{a}tal \emph{\cite{chvatal}}\emph{)} There exists a constant $t_0$ such that every $t_0$-tough graph is hamiltonian.
\end{conjecture}

For a graph $H$, a graph $G$ is said to be $H$-free if it does not contain any induced subgraph isomorphic to $H$.
Partial results related to Conjecture \ref{cong-chv} have been obtained in various restricted classes of graphs \cite{Bauer2006}, such as planar graphs \cite{Gerlach}, claw-free graphs \cite{Broersma1}, co-comparability graphs \cite{Deogun}, chordal graphs \cite{Chen,Kabela}, $k$-trees $(k\geq 2)$ \cite{Broersma2}, $2K_2$-free graphs \cite{Broersma,Shan,Ota}, $(P_2\cup P_3)$-free graphs \cite{Shan2}.
In this paper, we deal with Conjecture \ref{cong-chv} in $(P_2\cup kP_1)$-free graphs, where $P_2\cup kP_1$ is the graph consisting of one edge and $k$ isolated vertices.
In this graph family, there are many results related to conjecture \ref{cong-chv}. Among them are the following theorems.

\begin{theorem}
\emph{(}Nikoghosyan \emph{\cite{Nikoghosyan}}\emph{)} Every $1$-tough $(P_2\cup P_1)$-free graph is hamiltonian.
\end{theorem}

\begin{theorem}\label{Li-1}
\emph{(}Li et al. \emph{\cite{Li}}\emph{)} Let $R$ be an induced subgraph of $P_4$, $P_3\cup P_1$ or $P_2\cup 2P_1$. Then, $R$-free $1$-tough graph on at least three vertices is hamiltonian.
\end{theorem}

\begin{theorem}\label{Hatfield}
\emph{(}Hatfield and Grimm \emph{\cite{Hatfield})} If $G$ is a $3$-tough $(P_2\cup 3P_1)$-free graph on at least three vertices, then $G$ is hamiltonian.
\end{theorem}

\begin{theorem}\label{ShiShan}
\emph{(}Shi and Shan \emph{\cite{Shi}}\emph{)}. Let $k\geq4$ be an integer and let $G$ be a $4$-tough and $2k$-connected $(P_2\cup kP_1)$-free graph. Then $G$ is hamiltonian.
\end{theorem}

Moreover, Shi and Shan \cite{Shi} proposed the following  conjecture.

\begin{conjecture}\label{conj-shi}
\emph{(}Shi and Shan \emph{\cite{Shi}}\emph{)}. Let $k\geq4$ be an integer and let $G$ be a $1$-tough and $2k$-connected $(P_2\cup kP_1)$-free graph. Then $G$ is hamiltonian.
\end{conjecture}

Recently, Xu et al. \cite{Xu} and Ota and Sanka \cite{Ota2} proved the following two results,  both of which confirms Conjecture \ref{conj-shi}.

\begin{theorem}\emph{(}Xu et al. \emph{\cite{Xu}}\emph{)}.\label{Xu}
Let $k\geq 1$ be an integer and let $G$ be a $1$-tough and $\max \{2 k-2,2\}$-connected $(P_2\cup kP_1)$-free graph. Then $G$ is hamiltonian.
\end{theorem}

\begin{theorem}\emph{(}Ota and Sanka \emph{\cite{Ota2}}\emph{)}.\label{Ota2}
Let $k\geq 2$ be an integer and let $G$ be a $1$-tough and $k$-connected $(P_2\cup kP_1)$-free graph with $\delta(G) \geq \frac{3(k-1)}{2}$. Then $G$ is hamiltonian or the Petersen graph.
\end{theorem}

A well-known theorem of Fan \cite{Fan} states that every $2$-connected non-complete graph $G$ with $\mu_2(G)\geq \frac{|G|}{2}$ is hamiltonian. In this paper, we  generalize Theorems \ref{Xu} and \ref{Ota2} by proving the following Fan-type theorem.

\begin{theorem}\label{thm-main}
Let $k\geq 2$ be an integer and let $G$ be a $1$-tough and $k$-connected $(P_2\cup kP_1)$-free graph. If $\mu_{k+1}(G)\geq\frac{7k-6}5$, then $G$ is hamiltonian or the Petersen graph.
\end{theorem}

Note that $\mu_{k+1}(G)\geq \delta(G)\geq \kappa(G)$.
Since a non-complete $t$-tough graph is $\lceil 2t\rceil$-connected,
as  corollaries of Theorem \ref{thm-main}, we have the following result, which generalizes Theorem \ref{Hatfield}.

\begin{corollary}\label{cor-main1}
Let $k\geq 3$ be an integer and let $G$ be a $\frac{7k-6}{10}$-tough  $(P_2\cup kP_1)$-free graph. Then, $G$ is hamiltonian or the Petersen graph.
\end{corollary}

Let $P^*$ be the Petersen graph. Note that if $k$ is an integer  such that $P^*$ is $k$-connected and $(P_2\cup kP_1)$-free, then $k=\kappa(P^*)=3$. Because $\alpha_e(P^*)>3$,  the following  corollary of Theorem \ref{thm-main} is true.

\begin{corollary}\label{cor-main2}
Let $k\geq 2$ be an integer and let $G$ be a $1$-tough and $k$-connected $(P_2\cup kP_1)$-free graph. If $\alpha_e(G)\leq k$, then $G$ is hamiltonian.
\end{corollary}

In a sense, Corollary \ref{cor-main2} can be seen as an extension of the  classical Chv\'{a}tal-Erd\H{o}s Theorem \cite{chvatal2} that every $k$-connected graph on at least three vertices is hamiltonian if it has independence number at most $k$.

\section{Preliminaries}
In this section, we show some properties of $(P_2\cup kP_1)$-free graphs, which is frequently used in the proof of Theorem \ref{thm-main}.

\begin{lem}
\label{lem-1}
Let $G$ be a $(P_2\cup kP_1)$-free graph and let $A, ~B$  be two independent sets of $G$. Then, the following statements are true:
\begin{wst}
\item[{\rm (i)}] If $|A\cap B|\geq k$, then $A\cup B$ is an independent set of $G$;
\item[{\rm (ii)}] $|N(x)\cap A|\geq |A|-k+1$ holds for all $x\in N(A)$.
\end{wst}
\end{lem}
\begin{proof}
For (i), by way of contradiction, assume that $A\cup B$ is not an independent set of $G$. Then there exist $x\in A\setminus B$ and $y\in B\setminus A$ such that $xy\in E(G)$. As $|A\cap B|\geq k$, there is a subset $S$ of $A\cap B$ with $|S|=k$. Then $G[\{x,y\}\cup S]\cong P_2\cup kP_1$, a contradiction. Hence, (i) is true.

For (ii), by way of contradiction, assume that (ii) is false. Then there exists $x\in N(A)$ such that $|N(x)\cap A|\leq |A|-k$. Set $B=(A-N(x))\cup \{x\}$. Then, $B$ is an independent set of $G$ such that $|A\cap B|=|A-N(x)|=|A|-|N(x)\cap A|\geq k$. By (i), we can deduce that $A\cup B$ is an independent set of $G$. This implies $x\notin N(A)$, a contradiction. Hence, (ii) is true.
\end{proof}
\smallskip

\begin{lem}
\label{lem-2}
Let $G$ be a $k$-connected $(P_2\cup kP_1)$-free graph. Let $C$ be a longest cycle of $G$, and let $H$ be a component of $G-V(C)$. Then, $|V(H)|=1$ and $N_C^-(H)\cup \{u_0\}$ is an essential independent set of $G$ with center $u_0$, where $u_0\in V(H)$.
\end{lem}
\begin{proof} For convenience, we let $x^+:=x_C^+$ for each $x\in V(C)$.
First, we claim that
\begin{equation}\label{eqn-1}
\text{$N_C^-(H)\cup \{u_0\}$ is an independent set for any $u_0\in V(H)$.}
\end{equation}
By way of contradiction, assume that $N_C^-(H)\cup \{u_0\}$ is not an independent set for some $u_0\in V(H)$. Then there exist $x,y\in N_C^-(H)\cup \{u_0\}$ such that $xy\in E(G)$. If $\{x,y\}\cap \{u_0\}=\emptyset$, then $x,y\in N_C^-(H)$ and so $x^+,y^+\in N_C(H)$. Let $P$ be a longest $(x^+,y^+)$-path in $G[V(H)\cup\{x^+,y^+\}]$. Then,
$$
y^+\overrightarrow{C}xy\overleftarrow{C}x^+\overrightarrow{P}y^+
$$
is a cycle longer than $C$, a contradiction. Hence, $\{x,y\}\cap \{u_0\}\neq \emptyset$. By renaming $x,y$ (if necessary), we may assume that $x=u_0$ and $y\in N_C^-(H)$. Then, $y^+\in N_C(H)$. Let $P'$ be a longest $(u_0,y^+)$-path in $G[V(H)\cup\{y^+\}]$. Then,  $yu_0\overrightarrow{P'}y^+\overrightarrow{C}y$
is a cycle longer than $C$, a contradiction. Hence, \eqref{eqn-1} is true.

Now, we claim that
\begin{equation}\label{eqn-2}
|V(H)|=1.
\end{equation}
By way of contradiction, assume that \eqref{eqn-2} is false. As $H$ is a component of $G-V(C)$, there exists an edge $uv\in E(H)$. As $G$ is $k$-connected, $|N^-_C(H)|\geq k$. Let $S$ be a $k$-subset of $N^-_C(H)$. By applying \eqref{eqn-1} with $u_0=u$ and $v$ respectively, we can deduce that both $N^-_C(H)\cup\{u\}$ and $N^-_C(H)\cup\{v\}$ are independent sets of $G$. It follows that
$G[\{u,v\}\cup S]\cong P_2\cup kP_1$, a contradiction. Hence, \eqref{eqn-2} is true.

It follows from \eqref{eqn-2} that $N_C(H)=N(u_0)$,  where $u_0$ is the unique vertex of $H$.  Together with \eqref{eqn-1}, we see that $dist_G(u_0, x)=2$ for all $x\in N_C^-(H)$, and hence $N_C^-(H)\cup \{u_0\}$ is an essential independent set of $G$ with center  $u_0$. Therefore, Lemma \ref{lem-2} is  true.
\end{proof}

\section{Proof of Theorem \ref{thm-main}}
Let $G$ be a $1$-tough and $k$-connected $(P_2\cup kP_1)$-free graph with $\mu_{k+1}(G)\geq\frac{7k-6}{5}$.  To prove Theorem \ref{thm-main}, we show that if $G$ is not hamiltonian, then it is the Petersen graph.

Since $G$ is not hamiltonian, $V(G)-V(C)\neq\emptyset$ holds for
any cycle $C$ in $G$.   Pick a longest cycle $C$ and a vertex $u_0\in V(G)-V(C)$ such that $d_G(u_0)$ is as large as possible.

Set $N_C(u_0):=\{x_1,x_2,\ldots,x_m\}$, where the vertices $x_1,x_2,\ldots,x_m$ appear in this order along $C$. By Lemma \ref{lem-2}, $V(G)-V(C)$ is an independent set in $G$, and hence $m=d_G(u_0)$.
Note that $m\geq k$ as $G$ is $k$-connected. For simplicity, we denote $X:=N_C(u_0)$, $X^-:=X_C^-$ and $X^+:=X_C^+$. For each $x\in V(C)$, let $x^-:=x_C^-$ and $x^{-2}:=(x^-)^-$. We break the proof of Theorem \ref{thm-main} into a series of claims, and prove them one by one.

\begin{claim}\label{cla-3.1}
$m\geq \mu_{k+1}(G)$.
\end{claim}
\begin{proof}
By Lemma \ref{lem-2}, $X^-\cup \{u_0\}$ is an essential independent set of $G$ with center $u_0$. As $m\geq k$, $\{u_0,x_1^-,\ldots,x_k^-\}$ is an essential independent set of $G$.
By the definition of $\mu_{k+1}(G)$, we have
\begin{equation}\label{eqn-3}
\max\{d_G(u_0), d_G(x_1^-), \ldots, d_G(x_k^-)\}\geq \mu_{k+1}(G).
\end{equation}
By way of contradiction, assume that Claim \ref{cla-3.1} is false. Then, $d_G(u_0)=m< \mu_{k+1}(G)$. By \eqref{eqn-3}, there exists some integer $i\in [1,k]$ such that
\begin{equation}\label{eqn-4}
d_G(x_i^-)\geq \mu_{k+1}(G)>d_G(u_0).
\end{equation}
If $x_i^{-2}\in N(u_0)$, then $C':=x_i\overrightarrow{C}x_{i}^{-2}u_0x_i$ is a longest cycle in $G$. Note that $x_i^-\in V(G)-V(C')$.   By the choice of $(C,u_0)$, we have $d_G(u_0)\geq d_G(x_i^-)$, contrary to \eqref{eqn-4}. Hence, $x_i^{-2}\notin N(u_0)$.
As $x_i^{-2}x_i^-\in E(G)$, $x_i^{-2}\in N(X^-\cup\{u_0\})$. By applying Lemma \ref{lem-1} (ii) with $(A,x):=(X^-\cup\{u_0\}, x_i^{-2})$, we can derive that
$$
|N(x_i^{-2})\cap (X^-\cup\{u_0\})|\geq |X^-\cup\{u_0\}|-k+1= m-k+2\geq 2.
$$
As $x_i^{-2}\notin N(u_0)$, there exists some integer $j\in [1,m]$ with $j\neq i$ such that $x_i^{-2}x_j^-\in E(G)$. Let $C'':=x_i\overrightarrow{C}x_j^-x_i^{-2}\overleftarrow{C}x_ju_0x_i$. Then $C''$ is a longest cycle in $G$. Note that $x_i^-\in V(G)-V(C'')$. By the choice of $(C,u_0)$, we have $d_G(u_0)\geq d_G(x_i^-)$, contrary to \eqref{eqn-4}. Hence, Claim \ref{cla-3.1} is true.
\end{proof}
\smallskip

\begin{claim}
\label{cla-3.2}
For each $v\in V(G)-V(C)$, the following statements are true:
\begin{enumerate}
  \item[\emph{(i)}] $X^-\cup\{v\}$ is an  independent set in $G$;
  \item[\emph{(ii)}] $X^+\cup\{v\}$ is an  independent set in $G$.
\end{enumerate}
\end{claim}
\begin{proof}  For (i), suppose to the contrary that  there exists  $v\in V(G)-V(C)$ such that $X^-\cup\{v\}$ is not an independent set in $G$. Then, $v\neq u_0$ and $v\in N(X^-)$, because according to Lemma \ref{lem-2}, $X^-\cup\{u_0\}$ is an independent set in $G$. By applying Lemma \ref{lem-1} (ii) with $(A,x):=(X^-\cup \{u_0\},v)$, we can derive that
\begin{equation}\label{eqn-5}
|N(v)\cap (X^-\cup \{u_0\})|\geq (m+1)-k+1\geq 2.
\end{equation}
However, by Lemma \ref{lem-2}, each component of $G-V(C)$ has order one. Hence,
$u_0v\notin E(G)$. This together with \eqref{eqn-5} implies that $|N(v)\cap X^-|\geq 2$, and hence $x_i^-, x_j^-\in N(v)$ holds for some $i,j\in [1,m]$ with $i\neq j$.
It follows that
$$
x_i\overrightarrow{C}x_j^{-}vx_i^-\overleftarrow{C}x_ju_0x_i
$$
is a cycle longer than $C$, a contradiction. Hence, (i) is true.
By considering $\overleftarrow{C}$, we see that (ii) is true. The proof is now complete.
\end{proof}
\smallskip

\begin{claim}\label{cla-3.3}
For each $xy\in E(C)$, the following statements are true:
\begin{enumerate}
  \item[\emph{(i)}] $N(X^-)\cap\{x,y\}\neq\emptyset$;
  \item[\emph{(ii)}] $N(X^+)\cap\{x,y\}\neq\emptyset$.
\end{enumerate}
\end{claim}
\begin{proof}
For (i), assume to the contrary that $N(X^-)\cap\{x,y\}=\emptyset$ for some  $xy\in E(C)$. If $x\in X^-$, then $y\in N(X^-)$, a contradiction. Hence, $x\notin X^-$. Similarly, $y\notin X^-$. It follows that $G[\{x,y\}\cup\{x_1^-,\ldots,x_k^-\}]\cong P_2\cup kP_1$, a contradiction. Therefore, (i) is true. By considering $\overleftarrow{C}$, we see that (ii) is true. The proof is now complete.
\end{proof}
\smallskip

\begin{claim}\label{cla-3.4}
There exists an edge $uv\in E(C)$ such that $u, v\in N(X^-)$.
\end{claim}
\begin{proof}
By way of contradiction, assume that Claim \ref{cla-3.4} is false. Then $|N(X^-)\cap\{u,v\}|\leq 1$ for all $uv\in E(C)$. Together with Claim \ref{cla-3.3}, we see that
$|N(X^-)\cap\{u,v\}|= 1$ holds for all $uv\in E(C)$.
Then, along $C$, the vertices on $C$ alternate between vertices in $N(X^-)$ and vertices in $V(G)\setminus N(X^-)$. Thus, we have
\begin{equation}\label{eqn-6}
|V(C)\cap N(X^-)|=|V(C)\setminus N(X^-)|.
\end{equation}
By Claim \ref{cla-3.2}, for each $v\in V(G)-V(C)$, we have $v\notin N(X^-)$, and hence
$N(X^-)\subseteq V(C)$.
This together with \eqref{eqn-6} implies that
$$
|N(X^-)|=|V(C)\cap N(X^-)|=|V(C)\setminus N(X^-)|
$$
If $V(G)-N(X^-)$ is an independent set of $G$, then $G-N(X^-)$ has exactly $|V(G)-N(X^-)|$ components. Thus,
$$
\omega(G-N(X^-))=|V(G)-N(X^-)|>|V(C)\setminus N(X^-)|=|N(X^-)|
$$
which implies that $G$ is not $1$-tough, a contradiction. Therefore, $V(G)-N(X^-)$ is not an independent set of $G$. Let $x,y$ be two vertices of $V(G)-N(X^-)$ such that $xy\in E(G)$. Set $A=X^-\cup\{x\}$ and $B=X^-\cup\{y\}$. Note that $x,y\in (V(G)-V(C))\cup (V(C)-N(X^-))$. By Claim \ref{cla-3.2}, we can derive that both $A$ and $B$ are independent sets of $G$. Note that $|A\cap B|\geq|X^-|\geq k$. By Lemma \ref{lem-1} (i), $A\cup B$ is an independent set of $G$, and hence $xy\notin E(G)$, a contradiction. Therefore, Claim \ref{cla-3.4} is true.
\end{proof}
\smallskip

By Claim \ref{cla-3.4}, there exists an edge $uv$ of $C$ such that $u,v\in N(X^-)$.
By symmetry, we may assume that $u=v^-$ and $uv\in E(x_m\overrightarrow{C}x_1)$. If $v=x_1$, then $u=x_1^-\notin N(X^-)$, a contradiction. Hence, $v\neq x_1$. Similarly, by $v\in N(X^-)$, we have $v\neq x_1^-$.
It follows that $|V(x_m\overrightarrow{C}x_1)|\geq 4$ and
\begin{equation}\label{eqn-7}
uv\in E(x_m\overrightarrow{C}x_1^{-2}),
\end{equation}
and hence $v\notin N(u_0)$.

\smallskip

\begin{claim}\label{cla-3-5}
For $z\in\{u,v\}\cup\{x_1^{-2},x_2^{-2},\ldots,x_m^{-2}\}$,
$$
|N(z)\cap X^-|\geq m-k+2-|N(u_0)\cap\{z\}|.
$$
\end{claim}
\begin{proof}
Let $z\in\{u,v\}\cup\{x_1^{-2},x_2^{-2},\ldots,x_m^{-2}\}$, then $z\in N(X^-)$. It follows from Lemma \ref{lem-2} that $X^-\cup\{u_0\}$ is an independent set of $G$.  By applying Lemma \ref{lem-1} (ii) with $(A,x):=(X^-\cup\{u_0\},z)$, we get
$$
|N(z)\cap (X^-\cup \{u_0\})|\geq |X^-\cup \{u_0\}|-k+1=m-k+2,
$$
and hence
\begin{eqnarray*}
|N(z)\cap X^-|&=&|N(z)\cap (X^-\cup \{u_0\})|-|N(z)\cap\{u_0\}|\\
&\geq& m-k+2-|N(u_0)\cap\{z\}|.
\end{eqnarray*}
Therefore, Claim \ref{cla-3-5} is true.
\end{proof}
\smallskip

Set $p:=|N(v)\cap X^-|$ and $q:=\min\{i\in [1,m]:~ x_i^-\in N(u)\}$.
Recalling that $v\notin N(u_0)$, we have $|N(u_0)\cap\{v\}|=0$.
By using Claim \ref{cla-3-5} with $z:=v$, we can derive that
\begin{equation}\label{eqn-8}
p=|N(v)\cap X^-|\geq m-k+2-|N(u_0)\cap\{v\}|=m-k+2.
\end{equation}
Similarly, by using Claim \ref{cla-3-5} with $z=u$, we obtain that
\begin{equation}\label{eqn-9}
|[q,m]|\geq |N(u)\cap X^-|\geq m-k+2-|N(u_0)\cap\{u\}|\geq m-k+1.
\end{equation}
Let $x_{i_1}^-,x_{i_2}^-,\ldots,x_{i_p}^-$ be all neighbors of $v$ on $X^-$ such that $1\leq i_1<i_2<\cdots<i_p\leq m$.

\smallskip

\begin{claim}\label{cla-3-6}
$i_p<q$.
\end{claim}

\begin{proof}
By way of contradiction, assume that $i_p\geq q$. If $i_p>q$, then
$$
x_{i_p}\overrightarrow{C}ux_q^-\overleftarrow{C}vx_{i_p}^-\overleftarrow{C}x_qu_0x_{i_p}
$$
is a cycle longer than $C$, a contradiction. Hence, $i_p=q$, which implies that $u,v\in N(x_q^-)$.
It follows from \eqref{eqn-8} that $p\geq m-k+2\geq 2$, and hence $q=i_p\geq i_2\geq 2$. By using Claim \ref{cla-3-5} with $z:=x_q^{-2}$, we have
\begin{equation}\label{eqn-10}
|N(x_q^{-2})\cap X^-|\geq m-k+2-|N(u_0)\cap\{x_q^{-2}\}|.
\end{equation}
If $x_q^{-2}\in N(u_0)$, then $x_q^{-2}=x_{q-1}$.
Recalling that $q\geq 2$ and $v\in V(x_m^+Cx_1^{-2})$, we have $uv\in E(x_q\overrightarrow{C}x_q^{-2})$, and hence
$x_q\overrightarrow{C}ux_q^-v\overrightarrow{C}x_q^{-2}u_0x_q$ is a cycle longer than $C$, a contradiction. Thus, $x_q^{-2}\notin N(u_0)$. This together with \eqref{eqn-10} implies that
$|N(x_q^{-2})\cap X^-|\geq m-k+2\geq 2$.
Let $i$ be an integer with $i\in [1,m]\setminus\{q\}$ such that $x_i^-\in N(x_q^{-2})$.
Then,
$$
C':=x_i\overrightarrow{C}x_q^{-2}x_i^{-}\overleftarrow{C}x_qu_0x_i
$$
is a hamiltonian cycle in $G[(V(C)-\{x_q^-\})\cup\{u_0\}]$. As $q\geq 2$, $uv$ is an edge of $C'$. By replacing $uv$ with $ux_q^-v$ in $C'$, we get  a cycle longer than $C$, a contradiction. This completes the proof of Claim \ref{cla-3-6}.
\end{proof}

\smallskip

An $(x_q,y)$-path is \emph{good}  if it is a hamiltonian path in $G[V(C)]$.
Let $P$ be a good $(x_q,y)$-path and let $z,w\in V(C)$. If $z\overrightarrow{P}w=z\overrightarrow{C}w$, then we call $C[z,w]$ a \emph{good segment} for $P$.

We will use the following claim frequently.

\smallskip

\begin{claim}\label{cla-3-7}
Let $P$ be a good $(x_q,y)$-path and let $C[z,w]$ be a good segment for $P$. Then, $y\notin N(u_0)$ and $N(y)\cap X^-\cap V(C[z,w))=\emptyset$.
\end{claim}
\begin{proof}
Suppose to the contrary that $y\in N(u_0)$. As $x_q\in N(u_0)$ and $V(P)=V(C)$, $x_q\overrightarrow{P}yu_0x_q$ is a cycle longer than $C$, a contradiction. Hence, $y\notin N(u_0)$. If $N(y)\cap X^-\cap V(C[z,w))\neq\emptyset$,
then there exists some integer $i\in [1,m]$ such that $x_i^-\in N(y)\cap V(C[z,w))$.
As $C[z,w]$ is a good segment for $P$, we have $(x_i^-)_P^+=(x_i^-)_C^+=x_i$, and hence
$x_q\overrightarrow{P}x_i^-y\overleftarrow{P}x_iu_0x_q$ is a cycle longer than $C$, a contradiction. Therefore, $N(y)\cap X^-\cap V(C[z,w))=\emptyset$. This completes the proof of Claim \ref{cla-3-7}.
\end{proof}

\smallskip

\begin{claim}
\label{cla-3-8}
Let $a$ be an integer with $a\in [1,m]$ such that $x_a^-\in N(v)$. Then,
$a<q$, $N(x_a^{-2})\cap X^-\subseteq\{x_i^-: i\in [a, q]\}$ and $|N(x_a^{-2})\cap X^-|\geq m-k+2$.
\end{claim}
\begin{figure}[h]
\begin{center}
\psfrag{A}{$x_a^{-2}$}
\psfrag{B}{$x_a^-$}
\psfrag{C}{$x_a$}
\psfrag{D}{$x_q^-$}
\psfrag{E}{$x_q$}
\psfrag{I}{$\overrightarrow{C}$}
\psfrag{J}{$\overleftarrow{C}$}
\psfrag{U}{$u$}
\psfrag{V}{$v$}
\psfrag{Z}{$\overrightarrow{C}$}
\includegraphics[width=5cm]{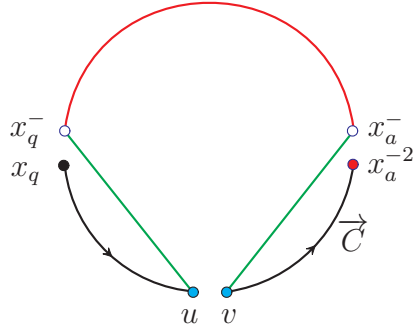}\\
\caption{A good $(x_q,x_a^{-2})$-path with good segments $C[x_q,u]$ and $C[v,x_a^{-2}]$.}
 \label{fig-1}
\end{center}
\end{figure}

\begin{proof}
Recall that $N(v)\cap X^-=\{x_{i_1}^-,\ldots,x_{i_p}^-\}$ and
$q:=\min\{i\in [1,m]:~ x_i^-\in N(u)\}$.
By Claim \ref{cla-3-6}, we have $a\leq i_p<q$.
Let $P:=x_q\overrightarrow{C}ux_q^-\overleftarrow{C}x_a^-v\overrightarrow{C}x_a^{-2}$ (see Fig. \ref{fig-1}). Then $P$ is a good $(x_q,x_a^{-2})$-path  with good segments $C[x_q,u]$ and $C[v,x_a^{-2}]$.
By applying Claim \ref{cla-3-7} with $y:=x_a^{-2}$, we have  $x_a^{-2}\notin N(u_0)$ and $$
N(x_a^{-2})\cap X^-\cap (V(C[x_q,u))\cup V(C[v,x_a^{-2})))=\emptyset.
$$
It follows that $N(x_a^{-2})\cap X^-\subseteq\{x_i^-: i\in [a, q]\}$. On the other hand, by Claim \ref{cla-3-5}, we have $|N(x_a^{-2})\cap X^-|\geq m-k+2-|N(u_0)\cap \{x_a^{-2}\}|$. Together with $x_a^{-2}\notin N(u_0)$, we can derive that $|N(x_a^{-2})\cap X^-|\geq m-k+2$. This completes the proof of Claim \ref{cla-3-8}.
\end{proof}

\smallskip

\begin{claim}
\label{cla-3-9}
Let $a,b$ be two integers with $a,b\in [1,m]$ and $a\neq b$. If $x_a^-\in N(v)$ and $x_b^-\in N(x_a^{-2})$, then
$a<b\leq q$, $N(x_b^{-2})\cap X^-\subseteq\{x_i^-:~i\in [1,a)\cup [b,q]\}$ and $|N(x_b^{-2})\cap X^-|\geq m-k+2$.
\end{claim}

\begin{figure}[h]
\begin{center}
\psfrag{A}{$x_a^{-2}$}
\psfrag{B}{$x_a^{-}$}
\psfrag{C}{$x_a$}
\psfrag{D}{$x_q^{-}$}
\psfrag{E}{$x_q$}
\psfrag{L}{$x_b$}
\psfrag{M}{$x_b^{-2}$}
\psfrag{N}{$x_b^{-}$}
\psfrag{U}{$u$}
\psfrag{V}{$v$}
\psfrag{Z}{$\overrightarrow{C}$}
\includegraphics[width=5cm]{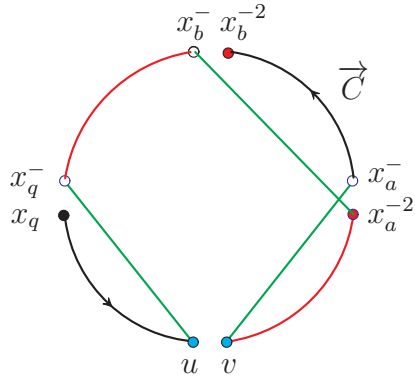}\\
\caption{A good $(x_q,x_b^{-2})$-path with good segments $C[x_q,u]$ and  $C[x_a^-,x_b^{-2}]$.}
 \label{fig-2}
\end{center}
\end{figure}

\begin{proof}
As $x_a^-\in N(v)$, by Claim \ref{cla-3-8}, we have $N(x_a^{-2})\cap X^-\subseteq\{x_a^-, x_{a+1}^-,\ldots,x_q^-\}$. This together with $x_b^-\in N(x_a^{-2})$ and $a\neq b$ implies that $a< b\leq q$.

Let $P:=x_q\overrightarrow{C}ux_q^-\overleftarrow{C}x_b^-x_a^{-2}\overleftarrow{C}
vx_a^-\overrightarrow{C}x_b^{-2}$ (see Fig. \ref{fig-2}). Then $P$ is a good $(x_q,x_b^{-2})$-path with good segments $C[x_q,u]$ and  $C[x_a^-,x_b^{-2}]$.
By applying Claim \ref{cla-3-7} with $y:=x_b^{-2}$, we can derive that $x_b^{-2}\notin N(u_0)$ and
$$
N(x_b^{-2})\cap X^-\cap (V(C[x_q,u))\cup V(C[x_a^-,x_b^{-2})))=\emptyset.
$$
It follows that $N(x_b^{-2})\cap X^-\subseteq\{x_i^-: i\in [1,a)\cup [b, q]\}$. On the other hand, by Claim \ref{cla-3-5}, we have $|N(x_b^{-2})\cap X^-|\geq m-k+2-|N(u_0)\cap \{x_b^{-2}\}|$. Together with $x_b^{-2}\notin N(u_0)$, we can derive that $|N(x_b^{-2})\cap X^-|\geq m-k+2$. This completes the proof of Claim \ref{cla-3-9}.
\end{proof}

\smallskip

\begin{claim}
\label{cla-3-10}
Let $a,b,c$ be three integers with $a,b,c\in [1,m]$ such that $x_a^-\in N(v)$, $x_b^-\in N(x_a^{-2})$ and $x_c^-\in N(x_b^{-2})$. If $q\geq c>b>a$, then
$N(x_c^{-2})\cap X^-\subseteq\{x_i^-:~i\in [a,b)\cup [c,q]\}$ and $|N(x_c^{-2})\cap X^-|\geq m-k+2$.
\end{claim}
\begin{figure}[h]
\begin{center}
\psfrag{A}{$x_a^{-2}$}
\psfrag{B}{$x_a^-$}
\psfrag{C}{$x_a$}
\psfrag{D}{$x_q^-$}
\psfrag{E}{$x_q$}
\psfrag{L}{$x_b$}
\psfrag{M}{$x_c^{-2}$}
\psfrag{N}{$x_c^{-}$}
\psfrag{O}{$x_c$}
\psfrag{P}{$x_b^{-2}$}
\psfrag{Q}{$x_b^-$}
\psfrag{J}{$a<c$}
\psfrag{U}{$u$}
\psfrag{V}{$v$}
\psfrag{Z}{$\overrightarrow{C}$}
\includegraphics[width=6cm]{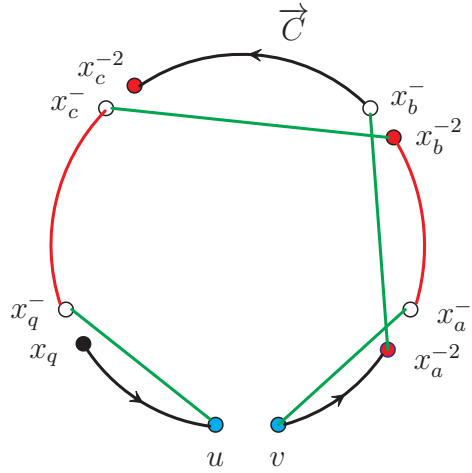}\\
\caption{A good $(x_q,x_c^{-2})$-path with good segments $C[x_q,u]$, $C[v,x_a^{-2}]$ and $C[x_b^-,x_c^{-2}]$.}
 \label{fig-3}
\end{center}
\end{figure}

\begin{proof}
Let $P:=x_q\overrightarrow{C}ux_q^-\overleftarrow{C}x_c^-x_b^{-2}\overleftarrow{C}
x_a^-v\overrightarrow{C}x_a^{-2}x_b^-\overrightarrow{C}x_c^{-2}$ (see Fig. \ref{fig-3}). Then $P$ is a good $(x_q,x_c^{-2})$-path with good segments $C[x_q,u]$, $C[v,x_a^{-2}]$ and $C[x_b^-,x_c^{-2}]$.
By applying Claim \ref{cla-3-7} with $y:=x_c^{-2}$, we can derive that $x_c^{-2}\notin N(u_0)$ and
$$
N(x_c^{-2})\cap X^-\cap (V(C[x_q,u))\cup V(C[v,x_a^{-2}))\cup V(C[x_b^-,x_c^{-2})))=\emptyset.
$$
It follows that $N(x_c^{-2})\cap X^-\subseteq\{x_i^-:~i\in [a,b)\cup [c,q]\}$. On the other hand, by Claim \ref{cla-3-5}, we have $|N(x_c^{-2})\cap X^-|\geq m-k+2-|N(u_0)\cap \{x_c^{-2}\}|$. Together with $x_c^{-2}\notin N(u_0)$, we can derive that $|N(x_c^{-2})\cap X^-|\geq m-k+2$. This completes the proof of Claim \ref{cla-3-10}.
\end{proof}

\smallskip

\begin{claim}\label{cla-3-11}
$i_p\leq \frac{3}{2}(m-k)+2$.
\end{claim}
\begin{proof}
Note that $x_{i_p}^-\in N(v)$. By using Claim \ref{cla-3-8} with $a=i_p$, we can derive that $i_p<q$ and
\begin{equation}\label{eqn-11}
|[i_p,q]|\geq |N(x_a^{-2})\cap X^-|\geq m-k+2.
\end{equation}
If $i_p>\frac{3}{2}(m-k)+2$, then by \eqref{eqn-11} and \eqref{eqn-9}, we can derive that
\begin{eqnarray*}
m+2&=&i_p+|[i_p,q]|+|[q,m]|\\
&>& \left[\frac{3}{2}(m-k)+2\right]+(m-k+2)+(m-k+1),
\end{eqnarray*}
and hence $m<\frac{7k-6}{5}$. On the other hand, by Claim \ref{cla-3.1}, we have
$m\geq\mu_{k+1}(G)\geq\frac{7k-6}{5}$, a contradiction. Therefore, Claim \ref{cla-3-11} is true.
\end{proof}

\smallskip

Recall that $N(v)\cap X^-=\{x_{i_1}^-,x_{i_2}^-,\ldots,x_{i_p}^-\}$, where $1\leq i_1<i_2<\cdots<i_p\leq m$. Set $h:=\lceil p/2\rceil$. It follows from \eqref{eqn-8} that $p=|N(v)\cap X^-|\geq m-k+2$. Hence,
\begin{equation}\label{eqn-12}
|[1,i_h]|=i_h\geq h=\lceil p/2\rceil\geq  p/2\geq (m-k+2)/2.
\end{equation}
By using Claim \ref{cla-3-8} with $a=i_h$, we have
\begin{equation}\label{eqn-13}
N(x_{i_h}^{-2})\cap X^-\subseteq \{x_i^-:~i\in [i_h,q]\}
\end{equation}
and
\begin{equation}\label{eqn-14}
|N(x_{i_h}^{-2})\cap X^-|\geq m-k+2.
\end{equation}

\smallskip

\begin{claim}\label{cla-3-12}
$N(v)\cap N(x_{i_h}^{-2})\cap X^-\nsubseteq\{x_{i_h}^-\}$.
\end{claim}
\begin{proof}
By way of contradiction, assume that  Claim \ref{cla-3-12} is false, then
$$
|N(v)\cap N(x_{i_h}^{-2})\cap X^-|\leq 1,
$$
and hence
$$
|(N(v)\cup N(x_{i_h}^{-2}))\cap X^-|\geq|N(v)\cap X^-|+|N(x_{i_h}^{-2})\cap X^-|-1.
$$
By combining this inequality with \eqref{eqn-8} and \eqref{eqn-14}, we get
\begin{equation}\label{eqn-15}
|(N(v)\cup N(x_{i_h}^{-2}))\cap X^-|\geq p+(m-k+2)-1\geq 2(m-k)+3.
\end{equation}
Denote $r:=\max~\{i:~ i\in [1,m], x_i^-\in N(v)\cup N(x_{i_h}^{-2})\}$. By \eqref{eqn-15}, we see that $r\geq 2(m-k)+3$.
Moreover, by Claim \ref{cla-3-11} and the definition of $N(v)\cap X^-$, we have
$$
i_p=\max~\{i:~ i\in [1,m], x_i^-\in N(v)\}\leq \frac{3}{2}(m-k)+2.
$$
Hence, $r>i_p>i_h$. Together with $x_r^-\in N(v)\cup N(x_{i_h}^{-2})$, we have
$x_r^-\in N(x_{i_h}^{-2})\setminus N(v)$. In particular, $x_{i_h}^{-2}\neq v$.
Recall that $x_{i_h}^-\in N(v)$.
By applying Claim \ref{cla-3-9} with $(a,b):=(i_h,r)$, we can derive that $r\leq q$ and
\begin{equation}\label{eqn-16}
|[1,i_h)\cup [r,q]|\geq |N(x_r^{-2})\cap X^-|\geq m-k+2.
\end{equation}
On the other hand, by \eqref{eqn-13} and the definition of $N(v)\cap X^-$, we have
$$
(N(v)\cup N(x_{i_h}^{-2}))\cap \{x_i^-:~i\in [1,i_h]\}=\{x_{i_j}^-:~j\in [1,h]\},
$$
Together with \eqref{eqn-15}, we can derive that
\begin{eqnarray*}
|(i_h,r]|&\geq& |(N(v)\cup N(x_{i_h}^{-2}))\cap X^-|-|(N(v)\cup N(x_{i_h}^{-2}))\cap \{x_i^-:~i\in [1,i_h]\}|\\
&\geq& (p+m-k+1)-h\\
&=&\left\lfloor p/2\right\rfloor+m-k+1.
\end{eqnarray*}
By combining this inequality with \eqref{eqn-9} and \eqref{eqn-16}, we get
\begin{eqnarray*}
m&=&|(i_h,r]|+|[q,m]|+|[1,i_h)\cup[r,q]|+|\{i_h\}|-|\{r,q\}|\\
&\geq& \left(\left\lfloor p/2\right\rfloor+m-k+1\right)+(m-k+1)+ (m-k+2)+1-2\\
&=&\left\lfloor p/2\right\rfloor+3m-3k+3.
\end{eqnarray*}
This together with \eqref{eqn-8} implies that
$$
m\geq \left\lfloor\frac{m-k+2}{2}\right\rfloor+3m-3k+3\geq \frac{7m-7k+7}{2}.
$$
It follows that $m\leq(7k-7)/5<\mu_{k+1}(G)$, contrary to Claim \ref{cla-3.1}. Hence, Claim \ref{cla-3-12} is true.
\end{proof}

\smallskip

\begin{claim}\label{cla-3-13}
There exists an integer $t\geq 0$ such that
$(k,m)=(5t+3,7t+3)$. Moreover, the following conditions hold:
\begin{enumerate}
  \item[\emph{(i)}] $(u,v)=(x_m,x_m^+)$;
  \item[\emph{(ii)}] $N(x_m)\cap X^-=\{x_{m-i}^-:~i\in [0,2t]\}$;
  \item[\emph{(iii)}] $|N(x_m^+)\cap X^-|=2t+2$ and $x_1^-,x_2^-\in N(x_m^+)$.
\end{enumerate}
\end{claim}
\begin{proof} It follows from Claim \ref{cla-3-12} that there is an integer $j\in [1,m]\setminus\{i_h\}$ such that $x_j^-\in N(v)\cap N(x_{i_h}^{-2})$.
By applying Claim \ref{cla-3-8} with $a:={i_h}$, we have $N(x_{i_h}^{-2})\cap X^-\subseteq\{x_i^-:~i\in [i_h,q]\}$, and hence
$j\in [i_h,q]$.
As $N(v)\cap X^-=\{x_{i_1}^-,x_{i_2}^-,\ldots,x_{i_p}^-\}$ and $j\neq i_h$,
\begin{equation}\label{eqn-17}
j\in \{i_s:~s\in[h+1,p]\}.
\end{equation}

Let $\ell:=~\max\{i\in [1,m]:x_i^-\in N(x_j^{-2})\}$. Note that $x_j^-\in N(v)$.
By using Claim \ref{cla-3-8} with $a:=j$, we can derive that $j<q$, $N(x_j^{-2})\cap X^-\subseteq\{x_i^-: i\in [j, q]\}$ and $|N(x_j^{-2})\cap X^-|\geq m-k+2$. Together with the definition of $\ell$, we see that $\ell\leq q$,
\begin{equation}\label{eqn-18}
N(x_j^{-2})\cap X^-\subseteq\{x_i^-: i\in [j, \ell]\}
\end{equation}
and
\begin{equation}\label{eqn-19}
|[j,\ell]|\geq |N(x_j^{-2})\cap X^-|\geq m-k+2.
\end{equation}
It follows from \eqref{eqn-17} and \eqref{eqn-19} that
$$
\ell=j-1+|[j,\ell]|\geq (j-1)+(m-k+2)>j\geq i_{h+1}>i_h.
$$
Hence, $i_h<j<\ell\leq q$.
Observe that $x_{i_h}^-\in N(v)$, $x_j^-\in
N(x_{i_h}^{-2})$ and $x_{\ell}^-\in N(x_j^{-2})$. By applying Claim \ref{cla-3-10}  with $(a,b,c):=(i_h,j,\ell)$, we have
\begin{equation}\label{eqn-20}
N(x_{\ell}^{-2})\cap X^-\subseteq\{x_i^-:~i\in [i_h,j)\cup [\ell,q]\}
\end{equation}
and
\begin{equation}\label{eqn-21}
|[{i_h},j)\cup [\ell, q]|\geq |N(x_{\ell}^{-2})\cap X^-|\geq m-k+2.
\end{equation}
By summing the inequalities in \eqref{eqn-9}, \eqref{eqn-12},\eqref{eqn-19} and \eqref{eqn-21}, we get
\begin{eqnarray}\label{eqn-22}
&&|[q,m]|+|[1,i_h]|+|[j,\ell]|+|[i_h,j)\cup [\ell, q]|\cr
&\geq&(m-k+1)+\frac{m-k+2}{2}+(m-k+2)+(m-k+2),
\end{eqnarray}
i.e., $m+3\geq 7(m-k)/2+6$, which is equivalent to  $m\leq (7k-6)/5$. On the other hand, by Claim \ref{cla-3.1}, we have $m\geq\mu_{k+1}(G)\geq (7k-6)/5$. Therefore, $m=\mu_{k+1}(G)=(7k-6)/5$,  and hence $(k,m)=(5t+3,7t+3)$ holds for some integer $t\geq 0$. Moreover, the inequality in  \eqref{eqn-22} achieves the equality, meaning that all inequalities in \eqref{eqn-9}, \eqref{eqn-12},\eqref{eqn-19} and \eqref{eqn-21} achieve the equalities. In the following, we will show that these equalities lead to conditions (i),(ii) and (iii).

Note that the equalities in \eqref{eqn-9} implies that $u\in N(u_0)$ and  $|[q,m]|=|N(u)\cap X^-|=m-k+1$. Together with \eqref{eqn-7} and the definition of $q$, we can derive that $q=k=m-2t$, $(u,v)=(x_m,x_m^+)$ and
$$
N(x_m)\cap X^-=N(u)\cap X^-=\{x_i^-: i\in [q,m]\}=\{x_{m-i}^-: i\in [0,2t]\}.
$$
Hence, both (i) and (ii) are correct.

Now, consider the equalities in \eqref{eqn-12}, which implies that
$i_h=h=p/2=(m-k+2)/2=t+1$. Together with (i) and the definition of $N(v)\cap X^-$, we have
$$
|N(x_m^+)\cap X^-|=|N(v)\cap X^-|=p=2t+2
$$
and
\begin{equation}\label{eqn-23}
N(x_m^+)=N(v)\supseteq \{x_{i_s}^-:~s\in [1,h]\}=\{x_i^-:~i\in [1,t+1]\}.
\end{equation}
Moreover, by \eqref{eqn-17}, $j\in \{i_s:~s\in[h+1,p]\}$. This together with  Claim \ref{cla-3-11} implies that
\begin{equation}\label{eqn-24}
j\in [i_{h+1},i_p]\subseteq \left[i_h+1,\frac{3}{2}(m-k)+2\right]=[t+2,3t+2].
\end{equation}
Recall that $v=x_m^+$ and $x_j^-\in N(v)\cap N(x_{i_h}^{-2})$. Hence, $x_j^-\in N(x_m^+)$.
Together with \eqref{eqn-23} and \eqref{eqn-24}, we can derive that $x_1^-,x_2^-\in N(x_m^+)$ whether $t=0$ or not. Therefore,  (iii) is true. This completes the proof of
Claim \ref{cla-3-13}.
\end{proof}

\smallskip

For $i\in [1,m]$, define $C_i:=x_i\overrightarrow{C}x_{i+1}$, where the indices are taken modulo $m$. We call $C_i$ a \emph{bad interval} of $C$ if  there exists  $xy\in E(C_i)$ such that $x,y\in N(X^-)$. By our assumption, $C_m$ is a bad interval of $C$.
In order to prove that each of $C_1, C_2,\ldots, C_{m-1}$ is also a bad interval of $C$, we need the following three claims, the first one of which follows from the symmetry of $C_i$ and $C_m$ on $C$.

\smallskip

\begin{claim}\label{cla-3-14}
Let $i\in [1,m]$ and let $zz^+\in E(C_i)$. If $z,z^+\in N(X^-)$, then there exists an integer $t\geq 0$ such that
$(k,m)=(5t+3,7t+3)$. Moreover, the following conditions hold:
\begin{enumerate}
  \item[\emph{(i)}] $(z,z^+)=(x_i,x_i^+)$;
  \item[\emph{(ii)}] $N(x_i)\cap X^-=\{x_{i-j}^-:~j\in [0,2t]\}$;
  \item[\emph{(iii)}] $|N(x_i^+)\cap X^-|=2t+2$ and $x_{i+1}^-, x_{i+2}^-\in N(x_i^+)$.
\end{enumerate}
\end{claim}

\smallskip

Noting that $(z,z^-,X^+)$ plays similar role in $\overleftarrow{C}$ as $(z,z^+,X^-)$ in $C$, we have the following analogy of Claim \ref{cla-3-14}.

\smallskip

\begin{claim}\label{cla-3-15}
Let $i\in [2,m+1]$ and let $z^-z\in E(C_{i-1})$. If $z^-,z\in N(X^+)$, then there exists an integer $t\geq 0$ such that
$(k,m)=(5t+3,7t+3)$. Moreover, the following conditions hold:
\begin{enumerate}
  \item[\emph{(i)}] $(z,z^-)=(x_i,x_i^-)$;
  \item[\emph{(ii)}] $N(x_i)\cap X^+=\{x_{i+j}^+:~j\in [0,2t]\}$;
  \item[\emph{(iii)}] $|N(x_i^-)\cap X^+|=2t+2$ and $x_{i-1}^+,x_{i-2}^+\in N(x_i^-)$.
\end{enumerate}
\end{claim}

\smallskip

\begin{claim}\label{cla-3-16}
For $i\in [1,m]$, the following statements are equivalent:
\begin{itemize}
  \item[\emph{(a)}] $C_i$ is a bad interval of $C$;
  \item[\emph{(b)}] $x_i^+\in N(X^-)$;
  \item[\emph{(c)}] $x_{i+1}^-\in N(X^+)$;
  \item[\emph{(d)}] $|V(C_i)|$ is even.
\end{itemize}
\end{claim}
\begin{proof} Note that $x_i\in N(X^-)$. By Claim \ref{cla-3-14}, we see that (a) is
equivalent to (b). In order to prove Claim \ref{cla-3-16}, it suffices to show that each of (b) and (c) is equivalent to (d).

Suppose $x_i^+\in N(X^-)$. By Claims \ref{cla-3.3} and \ref{cla-3-14}, we see that
\begin{equation}\label{eqn-25}
|\{x,y\}\cap N(X^-)|=1~\makebox{for all $xy\in E(x_i^+\overrightarrow{C}x_{i+1})$.}
\end{equation}
Then, along $C$, the vertices on $x_i^+\overrightarrow{C}x_{i+1}$ alternate between vertices in $N(X^-)$ and vertices in $V(G)\setminus N(X^-)$.
As $x_i^+,x_{i+1}\in N(X^-)$, $x_i^+\overrightarrow{C}x_{i+1}$ contains odd number of vertices, and hence $|V(C_i)|$ is even.

Suppose next that $|V(C_i)|$ is even. Then, $x_i^+\overrightarrow{C}x_{i+1}$ has odd number of vertices. Together with $x_{i+1}\in N(X^-)$ and \eqref{eqn-25}, we can derive that $x_i^+\in N(X^-)$.

From proofs above, we see that (b) is equivalent to (d). As $(x_{i+1},x_{i+1}^-,X^+)$ plays similar role in
$\overleftarrow{C}$ as $(x_i,x_i^+,X^-)$ in $C$, we also know that (c) is equivalent to (d). Hence, Claim \eqref{cla-3-16} is true.
\end{proof}

\smallskip

\begin{claim}\label{cla-3-17}
For $i\in [1,m]$, $C_i$ is a bad interval of $C$.
\end{claim}
\begin{proof} By our assumption, $C_m$ is a bad interval of $C$.
In order to prove Claim \ref{cla-3-17}, it suffices to show that for each  $i\in [1,m]$,
\begin{equation}\label{eqn-26}
\text{if $C_i$ is a bad interval of $C$, then so is $C_{i+1}$.}
\end{equation}
Suppose $C_i$ is a bad interval of $C$. Then, there exists an edge $zz^+\in E(C_i)$ such that $z,z^+\in N(X^-)$. By Claim \ref{cla-3-14}, we have $(z,z^+)=(x_i,x_i^+)$ and $x_{i+2}^-\in N(x_i^+)$, which means that  $x_{i+2}^-\in N(X^+)$.  By Claim \ref{cla-3-16}, we see that $C_{i+1}$ is also a bad interval of $C$. Hence, \eqref{eqn-26} is true. This completes the proof of Claim \ref{cla-3-17}.
\end{proof}

\smallskip

Let $i$ be an integer with $i\in [1,m]$. By Claim \ref{cla-3-17}, both $C_i$ and $C_{i-1}$ are  bad intervals of $C$. Together with Claim \ref{cla-3-16}, we see that $x_i^+\in N(X^-)$ and $x_i^-\in N(X^+)$. By applying  Claim \ref{cla-3-14} with $(z,z^+):=(x_i,x_i^+)$, we see that there is a  non-negative integer $t$ such that
\begin{equation}\label{eqn-27}
 (k,m)=(5t+3,7t+3),
\end{equation}
\begin{equation}\label{eqn-28}
 N(x_i)\cap X^-=\{x_{i-j}^-:~j\in [0,2t]\},
\end{equation}
\begin{equation}\label{eqn-29}
 |N(x_i^+)\cap X^-|=2t+2
\end{equation}
and
\begin{equation}\label{eqn-30}
x_{i+1}^-,x_{i+2}^-\in N(x_i^+).
\end{equation}
Similarly, by applying  Claim \ref{cla-3-15} with $(z,z^-):=(x_i,x_i^-)$, we have
\begin{equation}\label{eqn-31}
 N(x_i)\cap X^+=\{x_{i+j}^+:~j\in [0,2t]\}.
\end{equation}
If $t\geq 1$, then by \eqref{eqn-31}, we have $x_{i+1}^+\in N(x_i)$.  Together with \eqref{eqn-30}, we see that $x_{i+2}\overrightarrow{C}x_ix_{i+1}^+\overrightarrow{C}
x_{i+2}^-x_i^+\overrightarrow{C}x_{i+1}u_0x_{i+2}$
is a cycle longer than $C$, a contradiction. Hence, $t=0$.

It follows from \eqref{eqn-27}-\eqref{eqn-31} that $k=m=3$ and for each $i\in [1,3]$,
\begin{equation}\label{eqn-32}
N(x_i)\cap (X^-\cup X^+)=\{x_i^-,x_i^+\}
\end{equation}
and
\begin{equation}\label{eqn-33}
N(x_i^+)\cap X^-=\{x_{i+1}^-,x_{i+2}^-\},
\end{equation}
where the indices are taken modulo $3$. By symmetry, we also have
\begin{equation}\label{eqn-34}
N(x_i^-)\cap X^+=\{x_{i-1}^+,x_{i-2}^+\}, ~i=1,2,3.
\end{equation}

We claim that
\begin{equation}\label{eqn-35}
V(G)-V(C)=\{u_0\}.
\end{equation}
Suppose, to the contrary, that \eqref{eqn-35} is false. Then, there exists a vertex $u_0'\in (V(G)-V(C))-\{u_0\}$. By Claim \ref{cla-3.2}, both $X^-\cup\{u_0\}$ and $X^-\cup\{u_0'\}$ are independent sets of $G$. Note that $|(X^-\cup\{u_0\})\cap (X^-\cup\{u_0'\})|=|X^-|=k$. By Lemma \ref{lem-1} (i), we can derive that $X^-\cup\{u_0,u_0'\}$ is an independent set of $G$. On the other hand, by \eqref{eqn-33}, we have  $x_i^+\in N(X^-)$. By applying Lemma \ref{lem-1} (ii) with $(A,x)=(X^-\cup\{u_0,u_0'\}, x_i^+)$, we obtain that
$$
|N(x_i^+)\cap(X^-\cup\{u_0,u_0'\})|\geq  |X^-\cup\{u_0,u_0'\}|-k+1=3.
$$
This together with \eqref{eqn-33} implies that $N(x_i^+)\cap \{u_0,u_0'\}\neq\emptyset$, and hence at least one of $X^+\cup\{u_0\}$ and $X^+\cup\{u_0'\}$ is not an independent set of $G$, contrary to Claim \ref{cla-3.2}. Hence, \eqref{eqn-35} is true.

We claim next that
\begin{equation}\label{eqn-36}
x_1^{+2}= x_2^-.
\end{equation}
By way of contradiction, assume that \eqref{eqn-36} is false. Then, $x_1^{+2}\neq x_2^-$. Together with \eqref{eqn-33}, we see that $N(x_1^+)\supseteq\{x_1,x_1^{+2},x_2^-,x_3^-\}$, and hence $d_G(x_1^+)\geq 4$. Note that
$X^+\cup\{u_0\}$ is an independent in $G$ and $x_1^{+2}\in N(X^+)$. By applying Lemma \ref{lem-1} (ii) with $(A,x)=(X^+\cup\{u_0\},x_1^{+2})$, we can derive that
\begin{equation}\label{eqn-37}
|N(x_1^{+2})\cap (X^+\cup\{u_0\})|\geq |X^+\cup\{u_0\}|-k+1=2.
\end{equation}
Recall that $C_1$ is a bad interval of $C$. By Claim \ref{cla-3-16}, $|V(C_1)|$ is even, and hence $x_1^{+2}\notin N(u_0)$. This together with \eqref{eqn-37} implies that
$\{x_2^+,x_3^+\}\cap N(x_1^{+2})\neq\emptyset$. Say $x_j^+\in N(x_1^{+2})$ for some $j\in\{2,3\}$. Set
$$
C'=x_1^{+2}x_j^+\overrightarrow{C}x_1u_0x_j\overleftarrow{C}x_1^{+2}.
$$
Then, $C'$ is a longest cycle in $G$ such that $x_1^+\in V(G)-V(C')$. By the choice of $(C,u_0)$, we have $d_G(u_0)\geq d_G(x_1^+)$, which implies that $3\geq 4$, a contradiction. Hence, \eqref{eqn-36} is true. By symmetry, we also have $x_2^{+2}= x_3^-$ and $x_3^{+2}= x_1^-$. It follows that $C=x_1x_1^+x_2^-x_2x_2^+x_3^-x_3x_3^+x_1^-x_1$  (see Fig. \ref{fig-4}).
\begin{figure}[h]
\begin{center}
\psfrag{A}{$x_1$}
\psfrag{B}{$x_1^+$}
\psfrag{C}{$x_2^-$}
\psfrag{D}{$x_2$}
\psfrag{E}{$x_2^+$}
\psfrag{F}{$x_3^-$}
\psfrag{G}{$x_3$}
\psfrag{H}{$x_3^+$}
\psfrag{I}{$x_1^-$}
\psfrag{J}{$u_0$}
\includegraphics[width=5cm]{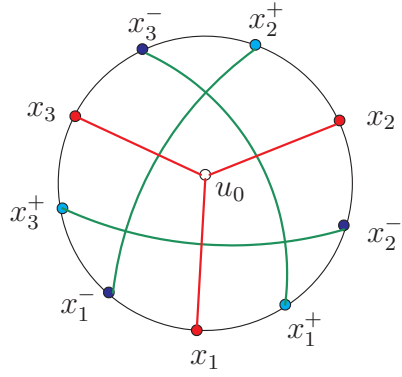}\\
\caption{The cycle $C$ and the graph $G$.}
 \label{fig-4}
\end{center}
\end{figure}

Finally,  we claim  that
\begin{equation}\label{eqn-38}
N(x_1)=\{u_0,x_1^+,x_1^-\}.
\end{equation}
For, otherwise, $d_G(x_1)\geq 4$ and  $C^*:=x_1^+x_3^-x_2^+x_1^-x_3^+x_3u_0x_2x_2^-x_1^+$ is a longest cycle of $G$ such that $(C^*,x_1)$ contradicts the choice of $(C,u_0)$. Therefore, \eqref{eqn-38} is true. By symmetry, we also have
\begin{equation}\label{eqn-39}
N(x_i)=\{u_0,x_i^+,x_i^-\}, ~i=2,3.
\end{equation}
As both $X^+$ and $X^-$ are independent sets of $G$,  by \eqref{eqn-33}, \eqref{eqn-34}, \eqref{eqn-38} and \eqref{eqn-39}, we can derive that $E(G[V(C)])=E(C)\cup\{x_1^+x_3^-,x_2^+x_1^-, x_3^+x_2^-\}$. This together with $N(u_0)=\{x_1,x_2,x_3\}$ and \eqref{eqn-35} implies that $G$ is isomorphic to the Petersen graph. This completes the proof of Theorem \ref{thm-main}.

\smallskip

\noindent{\bf Declaration of competing interest}

The authors declare that they do not have any commercial or associative interest that represents a conflict of interest in connection with the work submitted.

\end{document}